\documentclass[11pt]{article}

\usepackage{ascmac}
\usepackage{amsthm}
\usepackage{amsmath}
\usepackage{amssymb}
\usepackage{amsfonts}
\usepackage{bm}

\newtheorem{thm}{Theorem}[section]
\newtheorem{lem}[thm]{Lemma}

\makeatletter
\@addtoreset{equation}{section}

\makeatother

\title{On the $a$-points of the derivatives of the Riemann zeta function}
\author{TOMOKAZU ONOZUKA}

\begin{document}
\baselineskip=17pt
\date{}
\maketitle
\renewcommand{\thefootnote}{}
\renewcommand{\thefootnote}{\arabic{footnote}}
\setcounter{footnote}{0}
\begin{abstract}
We prove three results on the $a$-points of the derivatives of the Riemann zeta function. The first result is a formula of the Riemann-von Mangoldt type; we estimate the number of the $a$-points of the derivatives of the Riemann zeta function. The second result is on certain exponential sum involving $a$-points. The third result is an analogue of the zero density theorem. We count the $a$-points of the derivatives of the Riemann zeta function in $1/2-(\log\log T)^2/\log T<\Re s<1/2+(\log\log T)^2/\log T$.
\end{abstract}
\section{Introduction}

The Riemann zeta function $\zeta(s)$ is one of the most important functions in number theory, and its importance comes from its relation to the distribution of primes. The theory of the Riemann zeta function has a famous conjecture, which is the Riemann hypothesis. The Riemann hypothesis states that all of the nontrivial zeros of the Riemann zeta function are located on the critical line, $\Re s=1/2$. Hence it is important to study the zeros of the Riemann zeta function. In 1905, von Mangoldt proved the Riemann-von Mangoldt formula
$$
N(T)=\frac{T}{2\pi}\log\frac{T}{2\pi}-\frac{T}{2\pi}+O(\log T),
$$
where $N(T)$ is the number of zeros of the Riemann zeta function counted with multiplicity in the region $0<\Im s<T$. As a generalization of this formula, in 1913, Landau \cite{bo} estimated the number of the $a$-points of the Riemann zeta function, where we define the $a$-point of the function $f(s)$ as a root of $f(s)=a$. Especially, $\rho_a=\beta_a+i\gamma_a$ denotes the $a$-points of $\zeta(s)$. For $a\in\mathbb{C}$, he proved the following;
\begin{align}
\label{T1}N(a;1,T):=\sum_{1<\gamma_a<T}1=\begin{cases}
  \displaystyle\frac{T}{2\pi}\log\frac{T}{2\pi}-\frac{T}{2\pi}+O(\log T)&(a\neq1),\\
\\
  \displaystyle\frac{T}{2\pi}\log\frac{T}{4\pi}-\frac{T}{2\pi}+O(\log T)&(a=1).
\end{cases}
\end{align}

Furthermore, Landau \cite{la} proved another generalization of the Riemann-von Mangoldt formula. For $x>1$, he proved
\begin{align*}
\sum_{0<\gamma_0<T}x^{\rho_0}=-\Lambda(x)\frac{T}{2\pi}+O(\log T),
\end{align*}
where $\Lambda(x)$ is the von Mangoldt $\Lambda$ function if $x$ is an integer, and otherwise $\Lambda(x)=0$. If $x=1$, the left-hand side is just the number of zeros of the Riemann zeta function counted with multiplicity in the region $0<\Im s<T$, so this is a generalization of the Riemann-von Mangoldt formula.  This formula was also generalized by Steuding \cite{st} who proved that for any positive real number $x\neq1$ we have
\begin{align}
\label{T2}\sum_{0<\gamma_a<T}x^{\rho_a}=\left(\alpha(x)-x\Lambda\left(\frac{1}{x}\right)\right)\frac{T}{2\pi}+O(T^{\frac{1}{2}+\varepsilon}),
\end{align}
where $\alpha(x)$ is the coefficient of the series
\begin{align}
\frac{\zeta'(s)}{\zeta(s)-a}=\sum_{d>0}\frac{\alpha(d)}{d^s}\label{alp}
\end{align}
for $x\in\mathbb{Z}$ and $\alpha(x)=0$ for $x\notin\mathbb{Z}$ if $a\neq1$. If $a=1$,  $\alpha(x)$ is also the coefficient of (\ref{alp}) for $2^nx\in\mathbb{Z}$ with some $n\in\mathbb{N}$. If $2^nx\notin\mathbb{Z}$ for any $n\in\mathbb{N}$,  $\alpha(x)=0$.

On the other hand, in 1914, Bohr and Landau \cite{bl} showed that almost all zeros of the Riemann zeta function lie near the critical line. Landau \cite{bo} generalized this result under the Riemann hypothesis. He proved that almost all $a$-points of the Riemann zeta function lie  near the critical line under the Riemann hypothesis. Later, Levinson \cite{le} proved it unconditionally. Precisely, he proved that for sufficiently large $T$, $T^{1/2}\leq U \leq T$ and $a\in\mathbb{C}$, we have
\begin{align}
N^{(1)}(a;T,T+U):=\sum_{\substack{T<\gamma_a<T+U\\ \beta_a>1/2+(\log\log T)^2/\log T}}1=O\left(\frac{U\log T}{\log\log T}\right),\label{T3}\\
N^{(2)}(a;T,T+U):=\sum_{\substack{T<\gamma_a<T+U\\ \beta_a<1/2-(\log\log T)^2/\log T}}1=O\left(\frac{U\log T}{\log\log T}\right)\label{T4}
\end{align}
and
\begin{align}
\notag N^{(3)}(a;T,T+U)&:=\sum_{\substack{T<\gamma_a<T+U\\ 1/2-(\log\log T)^2/\log T<\beta_a<1/2+(\log\log T)^2/\log T}}1\\
&\qquad\qquad\qquad\qquad\qquad\quad=\frac{U}{2\pi}\log T+O\left(\frac{U\log T}{\log\log T}\right).\label{T5}
\end{align}

These results have been generalized to the case of the zeros of the derivatives of the Riemann zeta function. We define $\zeta^{(k)}(s)$ as the $k$th derivative of the Riemann zeta function, and $\rho_a^{(k)}=\beta_a^{(k)}+i\gamma_a^{(k)}$ denotes the $a$-point of $\zeta^{(k)}(s)$. In 1970, Berndt \cite{be} showed a formula of the Riemann-von Mangoldt type for $\zeta^{(k)}(s)$ with $k\geq1$;
\begin{align}
N_k(0;0,T):=\sum_{0<\gamma_0^{(k)}<T}1=\frac{T}{2\pi}\log\frac{T}{4\pi}-\frac{T}{2\pi}+O(\log T).\label{berndt}
\end{align}
Equations (\ref{T3})-(\ref{T5}) are also generalized to the zeros of  $\zeta^{(k)}(s)$ by Levinson and Montgomery \cite[Theorem 2]{lm}. They proved
\begin{align*}
N_k^{(1)}(0;0,T)+N_k^{(2)}(0;0,T):=\sum_{\substack{0<\gamma_0^{(k)}<T\\ |\beta_0^{(k)}-1/2|\geq\delta}}1\ll\delta^{-1}T\log\log T
\end{align*}
for $k\geq1$.

In this paper, we generalize these results. We generalize these estimations to the $a$-points of the derivatives of the Riemann zeta function.  In Section 2, we prove some lemmas and fundamental results on the $a$-points of $\zeta^{(k)}(s)$. Precisely, we find the "$a$-point free" region and the trivial $a$-points. In Section 3, we prove a generalization of Landau's result (\ref{T1}) and Berndt's result (\ref{berndt}).
\begin{thm}\label{1}
For any positive integer $k$ and any complex number $a\neq0$, we have
\begin{align*}
N_k(a;1,T):=\sum_{1<\gamma_a^{(k)}<T}1=\frac{T}{2\pi}\log\frac{T}{2\pi}-\frac{T}{2\pi}+O(\log T).
\end{align*}
\end{thm}
The reason why the summation does not count the $a$-points with $0<\gamma_a^{(k)}\leq1$ is that there exist many trivial $a$-points. See Theorem \ref{thm3}.

In Section 4, we show a generalization of Steuding's result (\ref{T2}).
\begin{thm}\label{2}
Let $x>1$. For any positive integer $k$ and any complex number $a$, we have
\begin{align*}
&\sum_{1<\gamma_a^{(k)}<T}x^{\rho_a^{(k)}}\\
&=\begin{cases}
  \displaystyle\frac{T}{2\pi}\sum_{\substack{l\geq0\\n_0,\ldots,n_l\geq2\\x=n_0\cdots n_l}}\frac{(-1)^{k(l+1)}}{a^{l+1}}(\log n_0)^{k+1}(\log n_1\cdots\log n_l)^k+O(\log T)&(a\neq0),\\
  \displaystyle\frac{T}{2\pi}\sum_{\substack{l\geq0\\n_0\geq2\\n_1,\ldots,n_l\geq3\\x=n_0\cdots n_l/2^{l+1}}}\left(\frac{-1}{(\log 2)^k}\right)^{l+1}(\log n_0)^{k+1}(\log n_1\cdots\log n_l)^k+O(\log T)&(a=0).
\end{cases}
\end{align*}
If $a\neq0$, the summation of the right-hand side is zero for $x\notin\mathbb{Z}$, and if $a=0$ and $2^nx\notin\mathbb{Z}$ for any $n\in\mathbb{N}$, the summation of the right-hand side is zero.
\end{thm}
Finally, in Section 5, we prove an analogue of Levinson's results (\ref{T3})-(\ref{T5}).
\begin{thm}\label{3}
Let $k$ be a positive integer, $\alpha>1/2$ be a real number and $a$ be a complex number. For sufficiently large $T$ and $T^\alpha\leq U \leq T$, we have
\begin{align*}
N_k^{(1)}(a;T,T+U)&:=\sum_{\substack{T<\gamma_a^{(k)}<T+U\\ \beta_a^{(k)}>1/2+(\log\log T)^2/\log T}}1=O\left(\frac{U\log T}{\log\log T}\right),\\
N_k^{(2)}(a;T,T+U)&:=\sum_{\substack{T<\gamma_a^{(k)}<T+U\\ \beta_a^{(k)}<1/2-(\log\log T)^2/\log T}}1=O\left(\frac{U\log T}{\log\log T}\right)
\end{align*}
and
\begin{align*}
&N_k^{(3)}(a;T,T+U):=\sum_{\substack{T<\gamma_a^{(k)}<T+U\\ 1/2-(\log\log T)^2/\log T\leq\beta_a^{(k)}\leq1/2+(\log\log T)^2/\log T}}1\\
&\qquad\qquad\qquad\qquad\qquad\qquad\qquad\qquad\qquad\qquad=\frac{U}{2\pi}\log T+O\left(\frac{U\log T}{\log\log T}\right).
\end{align*}
\end{thm}

These three main results imply a result on the uniform distribution on the derivatives of the Riemann zeta function. Actually, in \cite{lee}, Lee, Suriajaya and the author give the result on the uniform distribution of $\{\alpha\gamma_a^{(k)}\}_{\gamma_a^{(k)}>1}$ for all $\alpha\neq0$. (See \cite[THEOREM 1.1]{lee}.) In the proof of this result, all of main results in this paper are necessary, and play an important role.

The proofs of the main results are similar to the proofs of \cite[(22)]{bo}, \cite[Theorem 6]{st} and \cite[Theorem]{le}. The difference between those proofs and our proofs is the existence of the functional equation. In \cite{bo}, \cite{st} and \cite{le}, they studied the Riemann zeta function, and it has the functional equation. On the other hand, in this paper, we study the derivatives of the Riemann zeta function, and it does not have the functional equation. Instead of the functional equation, we prove and apply Lemma \ref{lem1} and Lemma \ref{LEM3.2}.


\section{Lemmas and Fundamental Results}

In this section, we prove some lemmas and fundamental results on the $a$-points for the derivatives of the Riemann zeta function. Hereafter we put $s=\sigma+it$.
\begin{lem}\label{lem1}
Let $k$ be a positive integer. For $c>1$, the following equation holds in the region $\{s\in\mathbb{C}\ |\ \sigma>c,|t|\geq1\}$;
\begin{align*}
\zeta^{(k)}(1-s)&=(-1)^k2(2\pi)^{-s}\Gamma(s)(\log s)^k\cos\frac{\pi s}{2}\ \zeta(s)\left(1+O\left(\frac{1}{|\log s|}\right)\right).\\
\end{align*}
\end{lem}
\begin{proof}
We use the equation \cite[(6)-(7c)]{sp}
\begin{align}
\zeta^{(k)}(1-s)&=(-1)^k2(2\pi)^{-s}\left\{\Gamma^{(k)}(s)\cos\frac{\pi s}{2}\ \zeta(s)+\sum_{j=0}^{k-1}\Gamma^{(j)}(s)R_{jk}(s)\right\},\label{L3}
\end{align}
where
\begin{align*}
R_{jk}(s)&=P_{jk}(s)\cos\frac{\pi s}{2}+Q_{jk}(s)\sin\frac{\pi s}{2},\\
P_{jk}(s)&=\sum_{n=0}^{k}a_{jkn}\zeta^{(n)}(s),\\
Q_{jk}(s)&=\sum_{n=0}^{k}b_{jkn}\zeta^{(n)}(s)
\end{align*}
and $a_{jkn},b_{jkn}$ are constants. By \cite[(10)]{sp}, derivatives of the gamma function can be estimated as
\begin{align}
\Gamma^{(j)}(s)=\Gamma(s)(\log s)^j\left(1+O\left(\frac{1}{s\log s}\right)\right).\label{L1}
\end{align}
By (\ref{L1}), we estimate the first and second terms of (\ref{L3}) in the region $\{s\in\mathbb{C}\ |\ \sigma>c,|t|\geq1\}$ as follows;
\begin{align}
\label{L4}\left|\Gamma^{(k)}(s)\cos\frac{\pi s}{2}\ \zeta(s)\right|&\asymp\left|\Gamma(s)(\log s)^ke^{\pi |t|/2}\right|,\\
\label{L5}\left|\sum_{j=0}^{k-1}\Gamma^{(j)}(s)R_{jk}(s)\right|&\ll\Gamma(s)(\log s)^{k-1}e^{\pi |t|/2},
\end{align}
since in the same region, we have $\zeta(s)\asymp1$ and $\zeta^{(j)}(s)=\sum_{n=2}^{\infty}(-\log n)^j/n^s\ll1$. Hence we obtain Lemma \ref{lem1}.
\end{proof}

By this lemma, we can find the "$a$-point free" region for $\zeta^{(k)}(s)$. When $a=0$, Spira \cite{sp}  found the zero free region for $\zeta^{(k)}(s)$. The next theorem is a generalization of his result.
\begin{thm}\label{thm2}
For any positive integer $k$ and $a\in\mathbb{C}$, there exist real numbers $E_{1k}(a)\leq0$ and $E_{2k}(a)\geq1$ such that $|\zeta^{(k)}(s)|> |a|$ for $\{s\in\mathbb{C}\ |\ \sigma\leq E_{1k}(a),|t|\geq1\}$ and $|\zeta^{(k)}(s)|< |a|$ for  $\{s\in\mathbb{C}\ |\ \sigma\geq E_{2k}(a)\}$. In particular, $\zeta^{(k)}(s)$ has no $a$-points for these two regions.
\end{thm}
\begin{proof}
When $a=0$, Titchmarsh \cite[Theorem 11.5(C)]{ti} and Spira \cite{sp} have already proved this theorem. Hence we only prove the case $a\neq0$. When $\sigma\leq E_{1k}(a)$ and $|t|\geq1$, it follows from Lemma \ref{lem1} that $\zeta^{(k)}(1-s)\to\infty$ as $\sigma\to\infty$, so $\zeta^{(k)}(1-s)\neq a$ for sufficiently large $\sigma$. When $\sigma\geq E_{2k}(a)$, since $\zeta^{(k)}(s)=\sum_{n=2}^{\infty}(-\log n)^k/n^s\to0$ as $\sigma\to\infty$, we have $\zeta^{(k)}(s)\neq a$ for sufficiently large $\sigma$.
\end{proof}
To state the next theorem, we define the region $\mathcal{C}_n$ as
\begin{align*}
\mathcal{C}_n:=\{s\in\mathbb{C}\ |\ -2n-1<\sigma<-2n+1,\ -1<t<1\}.
\end{align*}
Spira \cite{sp1} proved that there is an $\alpha_k$ such that $\zeta^{(k)}(s)$ has exactly one real zero in $\mathcal{C}_n$ for $1-2n\leq\alpha_k$. Levinson \cite{le} pointed out that $\zeta(s)=a$ has exactly one root in the neighborhood of $s=-2n$ for large $n$. The same phenomenon holds for $a$-points of $\zeta^{(k)}(s)$.
\begin{thm}\label{thm3}
For any positive integer $k$, there exists a positive integer $N=N_k(a)$ such that $\zeta^{(k)}(s)=a$ has just one root in $\mathcal{C}_n$ for each $n\geq N$. 
\end{thm}
\begin{proof}
We prove that there exists an integer $N=N_k(a)\in\mathbb{N}$ such that $\zeta^{(k)}(1-s)-a$ has just one zero in $\mathcal{C}'_n$ for each $n\geq N$, where 
\begin{align*}
\mathcal{C}'_n:=\{s\in\mathbb{C}\ |\ 2n<\sigma<2n+2,\ -1<t<1\}.
\end{align*}
By (\ref{L3}), we have
\begin{align*}
\zeta^{(k)}(1-s)-a=&\left\{(-1)^k2(2\pi)^{-s}\Gamma^{(k)}(s)\cos\frac{\pi s}{2}\ \zeta(s)\right\}\\
&\qquad+\left\{(-1)^k2(2\pi)^{-s}\sum_{j=0}^{k-1}\Gamma^{(j)}(s)R_{jk}(s)-a\right\}\\
=:&G_1(s)+G_2(s),
\end{align*}
say. By (\ref{L4}) and (\ref{L5}), there exists an $N_1$ such that $|G_1(s)|>|G_2(s)|$ holds in
$$
\{s\in\mathbb{C}\ |\ 2n\leq\sigma\leq2n+2,\ t=1\}
$$
for each $n\geq N_1$. Considering the complex conjugate, we find that there exists an $N_2$ such that $|G_1(s)|>|G_2(s)|$ holds in
$$
\{s\in\mathbb{C}\ |\ 2n\leq\sigma\leq2n+2,\ t=-1\}
$$
for each $n\geq N_2$. Next, we consider the segment
$$
\{s\in\mathbb{C}\ |\ \sigma=2n,\ -1\leq t\leq-1\}.
$$
Since $|\cos \pi s/2|\asymp1$ holds on this segment, we have
\begin{align*}
|G_1(s)|&\gg(2\pi)^{-\sigma}|\Gamma(s)||\log s|^k,\\
|G_2(s)|&\ll(2\pi)^{-\sigma}|\Gamma(s)||\log s|^{k-1}.
\end{align*}
Hence there exists an $N_3$ such that $|G_1(s)|>|G_2(s)|$ holds on this segment for each $n\geq N_3$.

Choosing $N=\max\{N_1,N_2,N_3\}$ and applying Rouch\'e's theorem, $\zeta^{(k)}(1-s)-a$ and $G_1(s)$ has the same number of zeros in $\mathcal{C}'_n$ for each $n\geq N$. The function $G_1(s)$ has just one zero $s=2n+1$ in $\mathcal{C}'_n$. Therefore, $\zeta^{(k)}(1-s)-a$ has just one zero in $\mathcal{C}'_n$.
\end{proof}

Note that if $a\in\mathbb{C}\setminus\mathbb{R}$, there are infinitely many $a$-points of $\zeta^{(k)}(s)$ in at least one of $\{s\in\mathbb{C}\ |\ 0<t<1\}$ or  $\{s\in\mathbb{C}\ |\ -1<t<0\}$, since $\zeta^{(k)}(s)\in\mathbb{R}$ for $s\in\mathbb{R}$. Furthermore if $a\in\mathbb{C}\setminus\mathbb{R}$, there exists infinitely many $a$-points, or infinitely many $\overline{a}$-points, of $\zeta^{(k)}(s)$ in $\{s\in\mathbb{C}\ |\ 0<t<1\}$ since  $\zeta^{(k)}(\overline{s})=\overline{\zeta^{(k)}(s)}$.

Next, to prove the main theorems, we prove the following three lemmas, Lemma \ref{LEM4}, Lemma \ref{LEM3.1} and Lemma \ref{LEM3.2}. These lemmas are generalizations of the classical results. The first lemma is a generalization of \cite[(2.12.6)]{ti} and \cite[(2.12.7)]{ti}.  The second lemma is a generalization of \cite[THEOREM 9.2]{ti}. The third lemma is a generalization of \cite[THEOREM 9.6 (A)]{ti}.

\begin{lem}\label{LEM4}
For any positive integer $k$ and any complex number $a$, there exist complex numbers $A_{k,a},\ B_{k,a}$ and a non-negative integer $m_{k,a}$ such that the following equations hold;
\begin{align*}
&(s-1)^{k+1}\left(\zeta^{(k)}(s)-a\right)=e^{A_{k,a}+B_{k,a}s}s^{m_{k,a}}\prod_{\rho^{(k)}_a\neq0}\left(1-\frac{s}{\rho^{(k)}_a}\right)e^{s/\rho^{(k)}_a},\\
&\frac{\zeta^{(k+1)}(s)}{\zeta^{(k)}(s)-a}=-\frac{k+1}{s-1}+B_{k,a}+\frac{m_{k,a}}{s}+\sum_{\rho^{(k)}_a\neq0}\left(\frac{1}{s-\rho^{(k)}_a}+\frac{1}{\rho^{(k)}_a}\right).
\end{align*}
\end{lem}
\begin{proof}
The second equation is given from the first equation by using the logarithmic derivative. Hence we prove the first equation. By Cauchy's integral theorem, we have
\begin{align}
\zeta^{(k)}(s)=\frac{k!}{2\pi i}\int_{|z-s|=\alpha}\frac{\zeta(z)}{(z-s)^{k+1}}dz\label{zeta2}
\end{align}
for positive $\alpha$. The Riemann zeta function is estimated as $\zeta(s)\ll|s|$ for $\sigma\geq1/2,|s-1|>1$ by \cite[(2.12.2)]{ti}. It follows from the functional equation $\zeta(1-s)=2(2\pi)^{-s}\cos(\pi s/2)\ \Gamma(s)\zeta(s)$ and Stirling's formula $|\Gamma(s)|\ll|\exp\{(s-1/2)\log s-s+O(1)\}|$ that we can also estimate the Riemann zeta function for $\sigma\leq1/2,|s|>1$ as
\begin{align}
\zeta(s)\ll\exp(|s|^{1+\varepsilon})\label{zeta}
\end{align}
with any small $\varepsilon>0$. Therefore (\ref{zeta}) holds for $|s|>2$. Using this estimation, we estimate (\ref{zeta2}) for $|s|>3$ with $\alpha=1$ as
\begin{align*}
\zeta^{(k)}(s)&\ll\exp(|2s|^{1+\varepsilon})\ll\exp(|s|^{1+2\varepsilon}).
\end{align*}

Since $\zeta(s)$ has only one simple pole at $s=1$ and the above estimation holds, $(s-1)^{k+1}(\zeta^{(k)}(s)-a)$ is an entire function and is of order $1$. Hence by the Hadamard factorization theorem, the lemma is valid.
\end{proof}

\begin{lem}\label{LEM3.1}
For any complex number $a$ and any sufficiently large $T$, we have
\begin{align*}
N_k(a;1,T+1)-N_k(a;1,T)\ll\log T.
\end{align*}
\end{lem}
\begin{proof}
First we prove that the estimation 
\begin{align}
\zeta^{(k)}(s)\ll |t|^{\mu(\sigma)+\varepsilon}\label{est}
\end{align}
holds as $|t|\to\infty$ for any small $\varepsilon>0$ and fixed $\sigma$ if $\mu(\sigma)$ satisfies $\zeta(s)\ll |t|^{\mu(\sigma)+\varepsilon}$. This function $\mu(\sigma)$ satisfies the inequality
\begin{align*}
\mu(\sigma)\leq\begin{cases}
  \displaystyle0&(\sigma\geq1)\\
  \displaystyle1/2-\sigma/2&(0<\sigma<1)\\
  \displaystyle1/2-\sigma&(\sigma\leq0)
\end{cases}
\end{align*}
by \cite[Section 5.1]{ti}. We use (\ref{zeta2}). In (\ref{zeta2}), $\zeta(z)$ can be estimated for fixed $\sigma$ as
\begin{align*}
\zeta(z)&\ll |t+\alpha|^{\mu(\sigma-\alpha)+\varepsilon}\\
&\ll |1+\alpha/t|^{\mu(\sigma)+\alpha+\varepsilon}\cdot|t|^{\mu(\sigma)+\alpha+\varepsilon}.
\end{align*}
Choosing $\alpha=1/\log t$, then $\zeta(z)\ll|t|^{\mu(\sigma)+\varepsilon}$ holds. Substituting this estimation into (\ref{zeta2}),  we have
\begin{align*}
\zeta^{(k)}(s)&\ll|\log t|^k|t|^{\mu(\sigma)+\varepsilon}\\
&\ll|t|^{\mu(\sigma)+k\log\log t/\log t+\varepsilon}.
\end{align*}
Since we have $k\log\log t/\log t\to 0$ as $t\to\infty$, the estimation (\ref{est}) holds.

When $a=0$, we can easily obtain Lemma \ref{LEM3.1} from (\ref{berndt}). Hence we only consider the case $a\neq0$. Since $\zeta^{(k)}(s)\to0$ as $\sigma\to0$, there exists a constant $C_1>E_{2k}(a)$ such that $|\zeta^{(k)}(C_1+it)|\leq |a|/2$ for all $t\in\mathbb{R}$. By Jensen's theorem, we have
\begin{align*}
&\int_0^{C_1-E_{1k}(a)+2}\frac{n(r)}{r}dr\\
&=\frac{1}{2\pi}\int_0^{2\pi}\log\left|\zeta^{(k)}\left(C_1+iT+\left(C_1-E_{1k}(a)+2\right)e^{i\theta}\right)-a\right|d\theta\\
&-\log\left|\zeta^{(k)}\left(C_1+iT\right)-a\right|,
\end{align*}
where $n(r)$ is the number of zeros of $\zeta^{(k)}(s)-a$ in the circle with center $C_1+iT$ and radius $r$. Since $\zeta^{(k)}(s)\ll |t|^{\mu(\sigma)+\varepsilon}$, there exists a constant $C_2>0$ such that
\begin{align*}
\log\left|\zeta^{(k)}\left(C_1+iT+(C_1-E_{1k}(a)+2)e^{i\theta}\right)-a\right|\leq C_2\log T.
\end{align*}
Furthermore since $|\zeta^{(k)}(C_1+iT)|\leq |a|/2$, we have
\begin{align*}
\log\left|\zeta^{(k)}\left(C_1+iT\right)-a\right|\ll 1.
\end{align*}
Hence we have
\begin{align}
\int_0^{C_1-E_{1k}(a)+2}\frac{n(r)}{r}dr\ll\log T.\label{nt1}
\end{align}
On the other hand, we have
\begin{align}
\notag&\int_0^{C_1-E_{1k}(a)+2}\frac{n(r)}{r}dr\\
\notag&\quad\geq\int_{C_1-E_{1k}(a)+1}^{C_1-E_{1k}(a)+2}\frac{n(r)}{r}dr\\
\label{nt2}&\quad\geq n(C_1-E_{1k}(a)+1)\int_{C_1-E_{1k}(a)+1}^{C_1-E_{1k}(a)+2}\frac{1}{r}dr.
\end{align}
From (\ref{nt1}) and (\ref{nt2}), we have
\begin{align*}
N_k(a;1,T+1)-N_k(a;1,T)\leq n(C_1-E_{1k}(a)+1)\ll\log T.
\end{align*}
\end{proof}

\begin{lem}\label{LEM3.2}
Let $\sigma_1$ and $\sigma_2$ be real numbers with $\sigma_1<\sigma_2$. For $s\in\mathbb{C}$ with $\sigma_1<\sigma<\sigma_2$ and large $t$, we have
\begin{align*}
\frac{\zeta^{(k+1)}(s)}{\zeta^{(k)}(s)-a}=\sum_{|\gamma_a^{(k)}-t|<1}\frac{1}{s-\rho_a^{(k)}}+O(\log t).
\end{align*}
\end{lem}
\begin{proof}
From Lemma \ref{LEM4}, we have
\begin{align}
\frac{\zeta^{(k+1)}(s)}{\zeta^{(k)}(s)-a}=\sum_{\rho^{(k)}_a\neq0}\left(\frac{1}{s-\rho^{(k)}_a}+\frac{1}{\rho^{(k)}_a}\right)+O(\log t).\label{A1}
\end{align}
We define $\mathcal{D}_1,\mathcal{D}_2,\mathcal{D}_3$ by
\begin{align*}
\mathcal{D}_1&:=\{s\in\mathbb{C}\ |\ E_{1k}(a)<\sigma<E_{2k}(a),\ 1<|t|\},\\
\mathcal{D}_2&:=\{s\in\mathbb{C}\ |\ \sigma<-2N_k(a)+1,\ |t|<1\},\\
\mathcal{D}_3&:=\{s\in\mathbb{C}\ |\ -2N_k(a)+1\leq\sigma\leq E_{2k}(a),\ |t|\leq1\},
\end{align*}
respectively. We divide the summation of (\ref{A1}) into the following three parts; 
\begin{align*}
\sum_{\rho^{(k)}_a\neq0}\left(\frac{1}{s-\rho^{(k)}_a}+\frac{1}{\rho^{(k)}_a}\right)&=\left(\sum_{\rho^{(k)}_a\in\mathcal{D}_1}+\sum_{\rho^{(k)}_a\in\mathcal{D}_2}+ \sum_{\substack{\rho^{(k)}_a\in\mathcal{D}_3\\\rho^{(k)}_a\neq0}}\right)\left(\frac{1}{s-\rho^{(k)}_a}+\frac{1}{\rho^{(k)}_a}\right)\\
&=:S_1(s)+S_2(s)+S_3(s),
\end{align*}
say. In the region $\mathcal{D}_3$, $\zeta^{(k)}(s)-a$ has only finitely many zeros, so we have $S_3(s)=O(1)$. By Theorem \ref{thm3}, every $a$-point is in $\mathcal{C}_n$, hence we have
\begin{align*}
|S_2|&\leq\sum_{n\geq N_k(a)}\left|\frac{1}{it+2n +O(1)}+\frac{1}{-2n +O(1)}\right|\\
&\ll\sum_{n\geq N_k(a)}\left|\frac{1}{it+2n}-\frac{1}{2n}\right|\\
&\ll\sum_{n\geq 1}\left(\frac{t}{4n^2+t^2}+\frac{t^2}{n(4n^2+t^2)}\right).
\end{align*}
The first term can be estimated as
\begin{align*}
\sum_{n\geq 1}\frac{t}{4n^2+t^2}\ll\sum_{1\leq n\leq t}\frac{1}{t}+\sum_{ n> t}\frac{t}{n^2}\ll1,
\end{align*}
and the second term can be estimated as
\begin{align*}
\sum_{n\geq 1}\frac{t^2}{n(4n^2+t^2)}\ll\sum_{1\leq n\leq t}\frac{1}{n}+\sum_{ n> t}\frac{t^2}{n^3}\ll\log t.
\end{align*}
Therefore we have $S_2(s)=O(\log t)$.

Finally we estimate (\ref{A1}). When $s=E_{2k}(a)+1+it$, since $|\zeta^{(k)}(E_{2k}(a)+1+it)|<|\zeta^{(k)}(E_{2k}(a))|<|a|$ holds by Theorem \ref{thm2}, we have
\begin{align*}
\frac{\zeta^{(k+1)}(E_{2k}(a)+1+it)}{\zeta^{(k)}(E_{2k}(a)+1+it)-a}\ll1.
\end{align*}
By the above argument, we can also estimate $S_2(E_{2k}(a)+1+it)\ll\log t$ and $S_3(E_{2k}(a)+1+it)\ll 1$. Hence we have
\begin{align*}
\frac{\zeta^{(k+1)}(s)}{\zeta^{(k)}(s)-a}=\sum_{\rho^{(k)}_a\in\mathcal{D}_1}\left(\frac{1}{s-\rho^{(k)}_a}-\frac{1}{E_{2k}(a)+1+it-\rho^{(k)}_a}\right)+O(\log t).
\end{align*}
For integer $n\neq-1,0$, by Lemma \ref{LEM3.1}, we have
\begin{align*}
&\sum_{t+n<\gamma_a^{(k)}\leq t+n+1}\left(\frac{1}{s-\rho^{(k)}_a}-\frac{1}{E_{2k}(a)+1+it-\rho^{(k)}_a}\right)\\
&\qquad=\sum_{t+n<\gamma_a^{(k)}\leq t+n+1}\frac{E_{2k}(a)+1-\sigma}{(s-\rho^{(k)}_a)(E_{2k}(a)+1+it-\rho^{(k)}_a)}\\
&\qquad\ll\sum_{t+n<\gamma_a^{(k)}\leq t+n+1}\frac{1}{(t-\gamma_a^{(k)})^2}\\
&\qquad\ll\sum_{t+n<\gamma_a^{(k)}\leq t+n+1}\frac{1}{n^2}\ll\frac{\log|t+n|}{n^2}.
\end{align*}
It follows from the above estimations that the following three estimations hold;
\begin{align*}
&\sum_{t+1<\gamma_a^{(k)}}\left(\frac{1}{s-\rho^{(k)}_a}-\frac{1}{E_{2k}(a)+1+it-\rho^{(k)}_a}\right)\ll\sum_{n=1}^{\infty}\frac{\log|t+n|}{n^2}\ll\log t,\\
&\sum_{1\leq\gamma_a^{(k)}<t-1}\left(\frac{1}{s-\rho^{(k)}_a}-\frac{1}{E_{2k}(a)+1+it-\rho^{(k)}_a}\right)\ll\sum_{1-t\leq n\leq -2}^{\infty}\frac{\log|t+n|}{n^2}+O(1)\ll\log t,\\
&\sum_{\gamma_a^{(k)}\leq-1}\left(\frac{1}{s-\rho^{(k)}_a}-\frac{1}{E_{2k}(a)+1+it-\rho^{(k)}_a}\right)\ll\sum_{n\leq-t-1}\frac{\log|t+n|}{n^2}+O(1)\ll1.
\end{align*}
Thus we have
\begin{align*}
\frac{\zeta^{(k+1)}(s)}{\zeta^{(k)}(s)-a}=\sum_{|\gamma_a^{(k)}-t|<1}\left(\frac{1}{s-\rho^{(k)}_a}-\frac{1}{E_{2k}(a)+1+it-\rho^{(k)}_a}\right)+O(\log t).
\end{align*}
The estimations
\begin{align*}
\sum_{|\gamma_a^{(k)}-t|<1}\frac{1}{E_{2k}(a)+1+it-\rho^{(k)}_a}\ll\sum_{|\gamma_a^{(k)}-t|<1}1\ll\log t
\end{align*}
are valid, so we obtain the lemma.
\end{proof}


\section{Proof of Theorem \ref{1}}

In this section, by applying Lemma \ref{lem1} and Lemma \ref{LEM3.2}, we prove Theorem \ref{1}.

($Proof$ $of$ $Theorem$ $\ref{1}$) 
Since $a\neq0$ and $\zeta^{(k)}(s)\to0$ as $\sigma\to\infty$, there exists a constant $E_{2k}'(a)\geq E_{2k}(a)$ such that
\begin{align}
\arg(\zeta^{(k)}(s)-a)\in(\arg (-a)-\pi/2,\arg (-a)+\pi/2)\label{TT1}
\end{align}
for all $s$ with $\sigma\geq E_{2k}'(a)$.
Furthermore, since it follows from Lemma \ref{lem1} that $\zeta^{(k)}(s)\to\infty$ holds as $\sigma\to-\infty$ for $t\geq1$, there exists a sufficiently small constant $E_{1k}'(a)\leq\min\{E_{1k}(a),-1\}$ such that $|a/\zeta^{(k)}(s)|<1$ for $\sigma\leq E_{1k}'(a)$ and $t\geq1$.
By the argument principle, we have
\begin{align*}
N_k(a;1,T)&=\frac{1}{2\pi}\Im\left(\int_{E_{1k}'(a)+i}^{E_{2k}'(a)+i}+\int_{E_{2k}'(a)+i}^{E_{2k}'(a)+iT}+\int_{E_{2k}'(a)+iT}^{E_{1k}'(a)+iT}+\int_{E_{1k}'(a)+iT}^{E_{1k}'(a)+i}\right)\frac{\zeta^{(k+1)}(s)}{\zeta^{(k)}(s)-a}ds\\
&=:\frac{1}{2\pi}(J_1+J_2+J_3+J_4),
\end{align*}
say. 

The first integral $J_1$ does not depend on $T$, so $J_1=O(1)$ holds.

The second integral is also estimated as $J_2=O(1)$ since we have
\begin{align*}
J_2=[\arg(\zeta^{(k)}(s)-a)]_{E_{2k}'(a)+i}^{E_{2k}'(a)+iT}\ll1
\end{align*}
by (\ref{TT1}).

Next we consider the fourth integral $J_4$. Since $|a/\zeta^{(k)}(s)|<1$, we have
\begin{align*}
J_4=\Im\int_{E_{1k}'(a)+iT}^{E_{1k}'(a)+i}\frac{\zeta^{(k+1)}(s)}{\zeta^{(k)}(s)}\left(1+\sum_{n=1}^{\infty}\left(\frac{a}{\zeta^{(k)}(s)}\right)^n\right)ds.
\end{align*}
By Lemma \ref{lem1} and Stirling's formula, we have
\begin{align}
\notag\zeta^{(k)}(E_{1k}'(a)+it)&\gg|t|^{1/2-E_{1k}'(a)}e^{-\pi |t|/2}\cdot\log(|t+2|)^k e^{\pi |t|/2}\\
&\gg|t|^{1/2-E_{1k}'(a)}\gg|t|^{3/2}.\label{hyoka}
\end{align}
Hence we have
\begin{align*}
\sum_{n=1}^{\infty}\left(\frac{a}{\zeta^{(k)}(E_{1k}'(a)+it)}\right)^n\ll|t|^{-3/2}.
\end{align*}
Moreover by Lemma \ref{lem1}, we have
\begin{align*}
\frac{\zeta^{(k+1)}(s)}{\zeta^{(k)}(s)}\ll|\log(1-s)|
\end{align*}
for $s=E_{1k}'(a)+it$. Therefore we have
\begin{align}
J_4&=\Im\int_{E_{1k}'(a)+iT}^{E_{1k}'(a)+i}\frac{\zeta^{(k+1)}(s)}{\zeta^{(k)}(s)}ds+O\left(\int_{E_{1k}'(a)+iT}^{E_{1k}'(a)+i}|\log(1-s)||t|^{-3/2}ds\right)\notag\\
&=\Im\int_{E_{1k}'(a)+iT}^{E_{1k}'(a)+i}\frac{\zeta^{(k+1)}(s)}{\zeta^{(k)}(s)}ds+O\left(1\right).\label{TT2}
\end{align}
The integral of the first term of (\ref{TT2}) coincides with $I_4$ in \cite{be}. By the calculations on $I_4$ in \cite{be}, we have
\begin{align*}
J_4=T\log\frac{T}{2\pi}-T+O(\log T).
\end{align*}

Finally, we consider the third term $J_3$. We apply Lemma \ref{LEM3.2}, then we have
\begin{align*}
J_3&=\Im\int_{E_{2k}'(a)+iT}^{E_{1k}'(a)+iT}\sum_{|\gamma_a^{(k)}-t|<1}\frac{1}{s-\rho_a^{(k)}}ds+O\left(\int_{E_{2k}'(a)+iT}^{E_{1k}'(a)+iT}\log tds\right)\\
&=\Im\sum_{|\gamma_a^{(k)}-T|<1}\int_{E_{2k}'(a)+iT}^{E_{1k}'(a)+iT}\frac{1}{s-\rho_a^{(k)}}ds+O\left(\log T\right).
\end{align*}
For each integral, we change the path of integration. If $\gamma_a^{(k)}\leq T$, then we change the path to the upper semicircle with center $\rho_a^{(k)}$ and radius $1$. If $\gamma_a^{(k)}> T$, then we change the path to the lower semicircle with center $\rho_a^{(k)}$ and radius $1$. Then we have
\begin{align*}
\int_{E_{2k}'(a)+iT}^{E_{1k}'(a)+iT}\frac{1}{s-\rho_a^{(k)}}ds\ll1.
\end{align*}
Therefore, by Lemma \ref{LEM3.1}, we have
\begin{align*}
J_3=\sum_{|\gamma_a^{(k)}-T|<1}O(1)+O\left(\log T\right)=O(\log T).
\end{align*}

Combining the estimations of $J_1, \ldots,J_4$, we obtain Theorem \ref{1}.  {\hfill $\square$}


\section{Proof of Theorem \ref{2}}

\begin{lem}\label{LEM4.1}
For $k\geq1$, $a\in\mathbb{C}$ and $s\in\mathbb{C}$ with sufficiently large $\sigma\geq E_{2k}(a)$, we have
\begin{align*}
&\frac{\zeta^{(k+1)}(s)}{\zeta^{(k)}(s)-a}\\
&=\begin{cases}
  \displaystyle\sum_{\substack{l\geq0\\n_0,\ldots,n_l\geq2}}\frac{(-1)^{k(l+1)}}{a^{l+1}}(\log n_0)^{k+1}(\log n_1\cdots\log n_l)^k\frac{1}{n_0^s\cdots n_l^s}&(a\neq0)\\
  \displaystyle\sum_{\substack{l\geq0\\n_0\geq2\\n_1,\ldots,n_l\geq3}}\left(\frac{-1}{(\log 2)^k}\right)^{l+1}(\log n_0)^{k+1}(\log n_1\cdots\log n_l)^k\frac{2^{(l+1)s}}{n_0^s\cdots n_l^s}&(a=0).
\end{cases}
\end{align*}
\end{lem}
\begin{proof}
When $a\neq0$, we have
\begin{align*}
\frac{\zeta^{(k+1)}(s)}{\zeta^{(k)}(s)-a}&=\frac{(-1)^{k+1}\sum_{n_0\geq 2}(\log n_0)^{k+1}/n_0^s}{-a(1-(-1)^ka^{-1}\sum_{n_1\geq2}(\log n_1)^k/n_1^s)}\\
&=\frac{(-1)^k}{a}\sum_{n_0\geq 2}\frac{(\log n_0)^{k+1}}{n_0^s}\sum_{l\geq0}\left(\frac{(-1)^k}{a}\sum_{n_1\geq2}\frac{(\log n_1)^k}{n_1^s}\right)^l\\
&=\sum_{\substack{l\geq0\\n_0,\ldots,n_l\geq2}}\frac{(-1)^{k(l+1)}}{a^{l+1}}(\log n_0)^{k+1}(\log n_1\cdots\log n_l)^k\frac{1}{n_0^s\cdots n_l^s}.
\end{align*}
When $a=0$, we have
\begin{align*}
\frac{\zeta^{(k+1)}(s)}{\zeta^{(k)}(s)}&=-\frac{\sum_{n_0\geq 2}(\log n_0)^{k+1}/n_0^s}{(\log 2)^k2^{-s}(1+(\log 2)^{-k}2^{s}\sum_{n_1\geq3}(\log n_1)^k/n_1^s)}\\
&=-\frac{2^s}{(\log 2)^k}\sum_{n_0\geq 2}\frac{(\log n_0)^{k+1}}{n_0^s}\sum_{l\geq0}\left(-\frac{2^s}{(\log 2)^k}\sum_{n_1\geq3}\frac{(\log n_1)^k}{n_1^s}\right)^l\\
&=\sum_{\substack{l\geq0\\n_0\geq2\\n_1,\ldots,n_l\geq3}}\left(\frac{-1}{(\log 2)^k}\right)^{l+1}(\log n_0)^{k+1}(\log n_1\cdots\log n_l)^k\frac{2^{(l+1)s}}{n_0^s\cdots n_l^s}.
\end{align*}
\end{proof}

($Proof$ $of$ $Theorem$ $\ref{2}$) 
For sufficiently large $U,V>0$ and $U>c>1$, by Cauchy's integral formula, we have
\begin{align*}
&\sum_{1<\gamma_a^{(k)}<T}x^{\rho_a^{(k)}}\\
&=\frac{1}{2\pi i}\left(\int_{-U+i}^{-c+i}+\int_{-c+i}^{V+i}+\int_{V+i}^{V+iT}+\int_{V+iT}^{-c+iT}+\int_{-c+iT}^{-U+iT}+\int_{-U+iT}^{-U+i}\right)x^s\frac{\zeta^{(k+1)}(s)}{\zeta^{(k)}(s)-a}ds\\
&=:\frac{1}{2\pi i}(K_1+K_2+K_3+K_4+K_5+K_6),
\end{align*}
say. Note that the constant $c$ is defined in the statement of Lemma \ref{lem1}.

The second term does not depend on $U$ and $T$, so we have $K_2=O(1)$.

Similar to the estimation on $J_3$, by Lemma \ref{LEM3.1} and Lemma \ref{LEM3.2} , the fourth term $K_4$ can be estimated as
\begin{align*}
K_4&=\sum_{|\gamma_a^{(k)}-T|<1}\int_{V+iT}^{-c+iT}x^s\frac{1}{s-\rho_a^{(k)}}ds+O\left(\int_{V+iT}^{-c+iT}x^{s}\log tds\right)\\
&=\sum_{|\gamma_a^{(k)}-T|<1}O(1)+O(\log T)\\
&=O(\log T).
\end{align*}

To estimate the fifth term $K_5$, we apply Lemma \ref{lem1} and we obtain
\begin{align}
\frac{\zeta^{(k+1)}(s)}{\zeta^{(k)}(s)-a}\ll|\log(1-s)|\label{A2}
\end{align}
for $\sigma<1-c$ and $|t|\geq1$. Thus we have
\begin{align*}
K_5&\ll\left|\int_{-c+iT}^{-U+iT}|x^{s}\log (1-s)||ds|\right|\\
&\ll\left|\int_{-c}^{-U}x^{\sigma}\log |1-\sigma-iT|d\sigma\right|\\
&\ll\log T.
\end{align*}
Note that this estimation does not depend on the choice of $U$.

By an argument similar to the above, we can estimate the first term $K_1$ as $K_1=O(1)$, and this estimation also does not depend on the choice of $U$.

By (\ref{A2}), we can estimate the sixth term $K_6$ as
\begin{align*}
K_6&\ll\left|\int_{-U+iT}^{-U+i}x^{-U}|\log(1-s)|ds\right|\\
&\ll x^{-U}T\log|1+U+iT|\\
&\to 0\quad (U\to\infty).
\end{align*}

Finally we estimate the third term $K_3$. In Lemma \ref{LEM4.1}, the right-hand side is too long, so we define $\alpha(d)$ as $\zeta^{(k+1)}(s)/(\zeta^{(k)}(s)-a)=\sum_{d}\alpha(d)d^{-s}$ for brevity. Then we have
\begin{align*}
K_3&=\int_{V+i}^{V+iT}x^s\sum_{d}\alpha(d)d^{-s}ds\\
&=\sum_{d}\alpha(d)\int_{V+i}^{V+iT}\left(\frac{x}{d}\right)^sds\\
&=Ti\alpha(x)+\sum_{d\neq x}\alpha(d)\left[\frac{(x/d)^s}{\log(x/d)}\right]_{V+i}^{V+iT}+O(1)\\
&=Ti\alpha(x)+O(1).
\end{align*}

Combining the estimations of $K_1,\ldots,K_6$, we obtain Theorem \ref{2}.   {\hfill $\square$}


\section{Proof of Theorem \ref{3}}

\begin{lem}\label{LEM5.1}
For $a\neq0$, $U\gg1$ and sufficiently large $T$, we have
\begin{align*}
2\pi \sum_{\substack{T<\gamma_a^{(k)}<T+U\\\beta_a^{(k)}>1/2}}\left(\beta_a^{(k)}-\frac{1}{2}\right)=\int_T^{T+U}\log\left|a-\zeta^{(k)}\left(\frac{1}{2}+it\right)\right|dt-U\log|a|+O(\log T).
\end{align*}
\end{lem}
\begin{proof}
For sufficiently large $C$ and a real number $b$ with $-b<C$, by Littlewood's lemma, we have
\begin{align}
\notag2\pi &\sum_{\substack{T<\gamma_a^{(k)}<T+U\\\beta_a^{(k)}>-b}}\left(\beta_a^{(k)}+b\right)\\
\notag&=\int_T^{T+U}\log\left|\frac{a-\zeta^{(k)}\left(-b+it\right)}{a}\right|dt-\int_T^{T+U}\log\left|\frac{a-\zeta^{(k)}\left(C+it\right)}{a}\right|dt\\
\label{Lit}&+\int_{-b}^C\arg\frac{a-\zeta^{(k)}(\sigma+i(T+U))}{a}d\sigma-\int_{-b}^C\arg\frac{a-\zeta^{(k)}(\sigma+iT)}{a}d\sigma
\end{align}
where we take the logarithmic branch of $\arg (1-\zeta^{(k)}(s)/a)$ as $\arg (1-\zeta^{(k)}(s)/a)\to0$ as $\sigma\to\infty$. We define the function $G_a(s)$ as
\begin{align*}
G_a(s):=\frac{a-\zeta^{(k)}(s)}{a}=1-\frac{\zeta^{(k)}(s)}{a}.
\end{align*}
Since $C$ is sufficiently large, we have $|G_a(s)|>1/2$ for $\sigma\geq C$. Furthermore we define $H_{a,T}(s)$ as
\begin{align*}
H_{a,T}(s):=\frac{G_a(s+iT)+G_{\overline{a}}(s-iT)}{2}.
\end{align*}
Let $n_z'(r)$ denote the number of zeros of $H_{a,T}(s)$  in the circle with center $z$ and radius $r$. Then we have $|\arg G_a(\sigma+iT)|\leq2\pi n_{C+iT}'(C+b)$ for $-b<\sigma<C$. Hence by Jensen's theorem and (\ref{est}), we have
\begin{align*}
|\arg G_a(\sigma+iT)|&\ll\int_0^{C+b+1}\frac{n_{C+iT}'(r)}{r}dr\\
&\ll\max_{0\leq \theta\leq 2\pi}\log\left|H_{a,T}\left(C+iT+\left(C+b+1\right)e^{i\theta}\right)\right|\\
&\ll\log T
\end{align*}
for $-b<\sigma<C$. Hence the fourth term of (\ref{Lit})  can be estimated as
\begin{align*}
\int_{-b}^C\arg\frac{a-\zeta^{(k)}(\sigma+iT)}{a}d\sigma\ll\log T.
\end{align*}
Considering $H_{a,T+U}(s)$, we can also estimate the third term of  (\ref{Lit}) as
\begin{align*}
\int_{-b}^C\arg\frac{a-\zeta^{(k)}(\sigma+i(T+U))}{a}d\sigma\ll\log T.
\end{align*}
Finally, we estimate the second term of  (\ref{Lit}). For $\sigma\geq C$, we have
\begin{align*}
|\zeta^{(k)}(s)|\leq\sum_{n=2}^{\infty}\frac{(\log n)^k}{n^\sigma}\ll2^{-\sigma+\varepsilon}
\end{align*}
for any $\varepsilon>0$. By this inequality and Cauchy's integral formula, we have
\begin{align*}
&\left|\int_T^{T+U}\log\left|\frac{a-\zeta^{(k)}\left(C+it\right)}{a}\right|dt\right|\\
&\leq\left|\int_T^{T+U}\log\frac{a-\zeta^{(k)}\left(C+it\right)}{a}dt\right|\\
&=\left|\int_C^{\infty}\left(\log\left(1-\frac{\zeta^{(k)}\left(\sigma+iT\right)}{a}\right)-\log\left(1-\frac{\zeta^{(k)}\left(\sigma+i(T+U)\right)}{a}\right)\right)d\sigma\right|\\
&\ll\int_C^{\infty}2^{-\sigma+\varepsilon}d\sigma\ll1.
\end{align*}
Hence we obtain
\begin{align}
2\pi \sum_{\substack{T<\gamma_a^{(k)}<T+U\\\beta_a^{(k)}>-b}}\left(\beta_a^{(k)}+b\right)=\int_T^{T+U}\log\left|a-\zeta^{(k)}\left(-b+it\right)\right|dt-U\log|a|+O(\log T).\label{b}
\end{align}
Taking $b=-1/2$, we obtain Lemma \ref{LEM5.1}.
\end{proof}
\begin{lem}\label{LEM5.2}
Let $\alpha>1/2$. For $T^\alpha\leq U\leq T$, we have
\begin{align*}
2\pi \sum_{\substack{T<\gamma_a^{(k)}<T+U\\\beta_a^{(k)}>1/2}}\left(\beta_a^{(k)}-\frac{1}{2}\right)\ll U\log\log T.
\end{align*}
\end{lem}
\begin{proof}
When $a=0$, Lemma \ref{LEM5.2} was already proved by Ki and Lee \cite[Theorem 3]{ki}. Hence we prove Lemma \ref{LEM5.2} in the case $a\neq0$. By Lemma \ref{LEM5.1}, it is enough to prove that there exists a constant $A$ such that
\begin{align}
\int_T^{T+U}\log\left|a-\zeta^{(k)}\left(\frac{1}{2}+it\right)\right|dt\leq AU\log\log T.\label{AUlogT}
\end{align}
By the triangle inequality, we have
\begin{align*}
&\int_T^{T+U}\log\left|a-\zeta^{(k)}\left(\frac{1}{2}+it\right)\right|dt\\
&\leq\int_T^{T+U}\log\left(|a|+\left|\zeta^{(k)}\left(\frac{1}{2}+it\right)\right|\right)dt\\
&\leq\frac{1}{v}\int_T^{T+U}\log\max\left(|a|^v,\left|\zeta^{(k)}\left(\frac{1}{2}+it\right)\right|^v\right)dt+U\log 2
\end{align*}
for positive $v$. By Jensen's inequality, we have
\begin{align*}
&\int_T^{T+U}\log\max\left(|a|^v,\left|\zeta^{(k)}\left(\frac{1}{2}+it\right)\right|^v\right)dt\\
&\leq U\log\left(\frac{1}{U}\int_T^{T+U}\max\left(|a|^v,\left|\zeta^{(k)}\left(\frac{1}{2}+it\right)\right|^v\right)dt\right)\\
&\leq U\log\left(\frac{1}{U}\int_T^{T+U}\left(|a|^v+\left|\zeta^{(k)}\left(\frac{1}{2}+it\right)\right|^v\right)dt\right)\\
&\leq U\log\left(|a|^v+\frac{1}{U}\int_T^{T+U}\left|\zeta^{(k)}\left(\frac{1}{2}+it\right)\right|^vdt\right).
\end{align*}
For positive $p,q$ with $1/p+1/q=1$, by H\"older's inequality, we have
\begin{align*}
&\int_T^{T+U}\left|\zeta^{(k)}\left(\frac{1}{2}+it\right)\right|^vdt\\
&\qquad\leq \left(\int_T^{T+U}\left|\frac{\zeta^{(k)}}{\zeta}\left(\frac{1}{2}+it\right)\right|^{pv}dt\right)^{1/p}\left(\int_T^{T+U}\left|\zeta\left(\frac{1}{2}+it\right)\right|^{qv}dt\right)^{1/q}.
\end{align*}
We put $v=2/(4k+1),\ p=1+1/(4k)$ and $q=4k+1$, then we have $pv=1/(2k),\ qv=2$ and $1/p+1/q=1$. By \cite[Claim]{ki} and \cite[Theorem 7.4]{ti}, we have
\begin{align*}
\int_T^{T+U}\left|\frac{\zeta^{(k)}}{\zeta}\left(\frac{1}{2}+it\right)\right|^{1/(2k)}dt\ll U\sqrt{\log T}
\end{align*}
and
\begin{align*}
\int_T^{T+U}\left|\zeta\left(\frac{1}{2}+it\right)\right|^2dt\ll U\log T.
\end{align*}
Hence we have
\begin{align*}
&\int_T^{T+U}\left|\zeta^{(k)}\left(\frac{1}{2}+it\right)\right|^vdt\ll \left(U\sqrt{\log T}\right)^{1/p}\left(U\log T\right)^{1/q}\ll U\log T.
\end{align*}
Thus we obtain (\ref{AUlogT}).
\end{proof}
\begin{lem}\label{LEM5.3}
Let $\alpha>1/2$. For $T^\alpha\leq U\leq T$, we have
\begin{align*}
N_k^{(1)}(a;T,T+U)=O\left(\frac{U\log T}{\log\log T}\right).
\end{align*}
\end{lem}
\begin{proof}
By Lemma \ref{LEM5.2}, we have
\begin{align*}
\sum_{\substack{T<\gamma_a^{(k)}<T+U\\ \beta_a^{(k)}>1/2+(\log\log T)^2/\log T}}\left(\beta_a^{(k)}-\frac{1}{2}\right)
\leq\sum_{\substack{T<\gamma_a^{(k)}<T+U\\ \beta_a^{(k)}>1/2}}\left(\beta_a^{(k)}-\frac{1}{2}\right)
\ll U\log\log T.
\end{align*}
On the other hand, we have
\begin{align*}
\sum_{\substack{T<\gamma_a^{(k)}<T+U\\ \beta_a^{(k)}>1/2+(\log\log T)^2/\log T}}\left(\beta_a^{(k)}-\frac{1}{2}\right)
\geq N_k^{(1)}(a;T,T+U)\frac{(\log\log T)^2}{\log T}.
\end{align*}
Combining these two inequalities, we obtain Lemma \ref{LEM5.3}.
\end{proof}
\begin{lem}\label{LEM5.4}
Let $a\neq0$ and $\alpha>1/2$. For sufficiently large $b$, $T$ and $T^\alpha\leq U\leq T$, we have
\begin{align*}
&2\pi\sum_{T<\gamma_a^{(k)}<T+U}\left(\beta_a^{(k)}+b\right)\\
&\qquad=\left(\frac{1}{2}+b\right)\left\{(T+U)\log\frac{T+U}{2\pi}-T\log\frac{T}{2\pi}-U\right\}\\
&\qquad+k\left\{(T+U)\log\log(T+U)-T\log\log T\right\}-U\log |a|+O\left(\frac{U}{\log T}\right).
\end{align*}
\end{lem}
\begin{proof}
We use (\ref{b}). The integrand can be calculated as
\begin{align}
\log\left|a-\zeta^{(k)}\left(-b+it\right)\right|=\log\left|\zeta^{(k)}\left(-b+it\right)\right|+\log\left|1-\frac{a}{\zeta^{(k)}\left(-b+it\right)}\right|.\label{b1}
\end{align}
By Lemma \ref{lem1}, the first term on the right-hand side of (\ref{b1}) is decomposed as
\begin{align*}
&\log\left|\zeta^{(k)}\left(-b+it\right)\right|\\
&=\log|\chi(-b+it)|+k\log|\log(1+b-it)|+\log|\zeta(1+b-it)|+O\left(\frac{1}{\log t}\right)
\end{align*}
where $\chi(s)=2^s\pi^{-1+s}\sin(\pi s/2)\Gamma(1-s)$. By (\ref{hyoka}), the second term on the right-hand side of (\ref{b1}) is estimated as
\begin{align*}
\log\left|1-\frac{a}{\zeta^{(k)}\left(-b+it\right)}\right|
\ll|t|^{-3/2}.
\end{align*}
Thus (\ref{b1}) can be calculated as
\begin{align*}
&\log\left|a-\zeta^{(k)}\left(-b+it\right)\right|\\
&=\log|\chi(-b+it)|+k\log|\log(1+b-it)|+\log|\zeta(1+b-it)|+O\left(\frac{1}{\log t}\right).
\end{align*}
Substituting this result into (\ref{b}), we have
\begin{align*}
&2\pi \sum_{\substack{T<\gamma_a^{(k)}<T+U\\\beta_a^{(k)}>-b}}\left(\beta_a^{(k)}+b\right)\\
&=\int_T^{T+U}\log|\chi(-b+it)|dt+k\int_T^{T+U}\log|\log(1+b-it)|dt\\
&+\int_T^{T+U}\log|\zeta(1+b-it)|dt+\int_T^{T+U}O\left(\frac{1}{\log t}\right)dt-U\log|a|+O(\log T)\\
&=:L_1+L_2+L_3+L_4-U\log|a|+O(\log T),
\end{align*}
say. For the first term $L_1$, by the equation written on the left side of \cite[p.1323]{le}, we have
\begin{align*}
L_1&=\int_T^{T+U}\left(\frac{1}{2}+b\right)\log\left|\frac{t}{2\pi}\right|dt+\int_T^{T+U}O\left(\frac{1}{t}\right)dt\\
&=\left(\frac{1}{2}+b\right)\left\{(T+U)\log\frac{T+U}{2\pi}-T\log\frac{T}{2\pi}-U\right\}+O(\log T).
\end{align*}
Next we consider the second term $L_2$. We have
\begin{align*}
L_2&=k\int_T^{T+U}\log\left|\log t+\log\left(-i+\frac{1+b}{t}\right)\right|dt\\
&=k\int_T^{T+U}\log|\log t+O(1)|dt\\
&=k\int_T^{T+U}\log\log tdt+\int_T^{T+U}O\left(\frac{1}{\log t}\right)dt\\
&=k(T+U)\log\log(T+U)-kT\log\log T+O\left(\frac{U}{\log T}\right).
\end{align*}
Similar to the estimation of the second term of  (\ref{Lit}), we can estimate $L_3=O(1)$. Furthermore we can easily estimate $L_4=O(U/\log T)$.
\end{proof}
\begin{lem}\label{LEM5.5}
For sufficiently large $b$, $T$ and $T^\alpha\leq U\leq T$, we have
\begin{align*}
&2\pi\sum_{T<\gamma_0^{(k)}<T+U}\left(\beta_0^{(k)}+b\right)\\
&\qquad=\left(\frac{1}{2}+b\right)\left\{(T+U)\log\frac{T+U}{2\pi}-T\log\frac{T}{2\pi}\right\}\\
&\qquad+k\left\{(T+U)\log\log(T+U)-T\log\log T\right\}\\
&\qquad-U\left(\frac{1}{2}+b+b\log2+k\log\log2\right)+O\left(\frac{U}{\log T}\right).
\end{align*}
\end{lem}
\begin{proof}
From \cite[Lemma 3.1]{lm}, we can easily obtain this result.
\end{proof}
\begin{lem}\label{LEM5.6}
For $T^\alpha\leq U\leq T$, any complex number $a$ and any positive integer $k$, we have
\begin{align*}
&N_k(a;T,T+U)=\begin{cases}
\displaystyle\frac{T+U}{2\pi}\log\frac{T+U}{2\pi}-\frac{T}{2\pi}\log\frac{T}{2\pi}-\frac{U}{2\pi}+O(\log T)&(a\neq0),\\
\\
\displaystyle\frac{T+U}{2\pi}\log\frac{T+U}{4\pi}-\frac{T}{2\pi}\log\frac{T}{4\pi}-\frac{U}{2\pi}+O(\log T)&(a=0).
\end{cases}
\end{align*}
\end{lem}
\begin{proof}
From Theorem \ref{1} and (\ref{berndt}), we can easily obtain this result.
\end{proof}

($Proof$ $of$ $Theorem$ $\ref{3}$) 
By Lemma \ref{LEM5.3}, $N_k^{(1)}(a;T,T+U)$ is already estimated. Hence we estimate $N_k^{(2)}(a;T,T+U)$ and $N_k^{(3)}(a;T,T+U)$. First we decompose the summation as
\begin{align*}
&2\pi \sum_{T<\gamma_a^{(k)}<T+U}\left(\beta_a^{(k)}+b\right)\\
&=2\pi \sum_{\substack{T<\gamma_a^{(k)}<T+U\\\beta_a^{(k)}>1/2+(\log\log T)^2/\log T}}\left\{\left(\beta_a^{(k)}-\frac{1}{2}\right)+\left(b+\frac{1}{2}\right)\right\}\\
&+2\pi \sum_{\substack{T<\gamma_a^{(k)}<T+U\\1/2-(\log\log T)^2/\log T\leq\beta_a^{(k)}\leq1/2+(\log\log T)^2/\log T}}\left\{\left(\beta_a^{(k)}-\frac{1}{2}\right)+\left(b+\frac{1}{2}\right)\right\}\\
&+2\pi \sum_{\substack{T<\gamma_a^{(k)}<T+U\\\beta_a^{(k)}<1/2-(\log\log T)^2/\log T}}\left(\beta_a^{(k)}+b\right).
\end{align*}
From Lemma \ref{LEM5.2}, we have
\begin{align*}
2\pi \sum_{T<\gamma_a^{(k)}<T+U}\left(\beta_a^{(k)}+b\right)&\leq O(U\log\log T)+2\pi \left(b+\frac{1}{2}\right)N_k^{(1)}(a;T,T+U)\\
&+2\pi \left(b+\frac{1}{2}\right)N_k^{(3)}(a;T,T+U)\\
&+2\pi \left(b+\frac{1}{2}-\frac{(\log\log T)^2}{\log T}\right)N_k^{(2)}(a;T,T+U).
\end{align*}
Since
\begin{align}
N_k^{(1)}(a;T,T+U)+N_k^{(3)}(a;T,T+U)= N_k(a;T,T+U)-N_k^{(2)}(a;T,T+U),\label{Rel}
\end{align}
we have
\begin{align*}
&2\pi \sum_{T<\gamma_a^{(k)}<T+U}\left(\beta_a^{(k)}+b\right)\\
&\qquad\leq 2\pi \left(b+\frac{1}{2}\right)N_k(a;T,T+U)\\
&\qquad-2\pi \frac{(\log\log T)^2}{\log T}N_k^{(2)}(a;T,T+U)+O(U\log\log T).
\end{align*}
By Lemma \ref{LEM5.6}, we have
\begin{align}
\notag&2\pi \sum_{T<\gamma_a^{(k)}<T+U}\left(\beta_a^{(k)}+b\right)\\
\notag&\qquad\leq \left(b+\frac{1}{2}\right)\left\{(T+U)\log(T+U)-T\log T\right\}\\
&\qquad-2\pi \frac{(\log\log T)^2}{\log T}N_k^{(2)}(a;T,T+U)+O(U\log\log T).\label{No1}
\end{align}
By Lemma \ref{LEM5.4} and Lemma \ref{LEM5.5}, the left-hand side is
\begin{align}
\notag&2\pi\sum_{T<\gamma_a^{(k)}<T+U}\left(\beta_a^{(k)}+b\right)\\
\notag&\qquad=\left(\frac{1}{2}+b\right)\left\{(T+U)\log(T+U)-T\log T\right\}\\
\label{No2}&\qquad+k\left\{(T+U)\log\log(T+U)-T\log\log T\right\}+O\left(U\right).
\end{align}
From (\ref{No1}) and (\ref{No2}), we have
\begin{align*}
&N_k^{(2)}(a;T,T+U)\\
&\quad\leq-\frac{k\log T}{2\pi(\log\log T)^2}\left\{(T+U)\log\log(T+U)-T\log\log T\right\}+O\left(U\frac{\log T}{\log\log T}\right)\\
&\quad\leq O\left(U\frac{\log T}{\log\log T}\right).
\end{align*}
Hence we have
\begin{align*}
&N_k^{(2)}(a;T,T+U)= O\left(U\frac{\log T}{\log\log T}\right).
\end{align*}

The estimation of $N_k^{(3)}(a;T,T+U)$ is given from (\ref{Rel}) and the estimations of $N_k^{(1)}(a;T,T+U)$ and $N_k^{(2)}(a;T,T+U)$. {\hfill $\square$}


Toyota Technological Institute\\
 2-12-1 Hisakata, Tempaku-ku, Nagoya 468-8511\\
 Japan\\
E-mail: onozuka@toyota-ti.ac.jp

\end{document}